\documentclass[12pt]{amsart}
\usepackage[all, cmtip]{xy}

\usepackage{times,amsfonts,amsmath,amstext,amsbsy,amssymb,
amsopn,amsthm,upref,eucal, mathrsfs, xcolor}

\usepackage{url}

\newtheorem{defn}{Definition}[section]
\newtheorem{thm}[defn]{Theorem}
\newtheorem{lem}[defn]{Lemma}
\newtheorem*{sub-lem}{Sub-lemma}

\newtheorem{cor}[defn]{Corollary}
\newtheorem{prop}[defn]{Proposition}
\newtheorem{rem}[defn]{Remark}

\newcommand{\R}{\mathbb{R}}
\newcommand{\C}{\mathbb{C}}
\newcommand{\Z}{\mathbb{Z}}

\newcommand{\T}{\mathbb{T}}
\newcommand{\D}{\mathbb{D}}

\newcommand{\esssup}{\text{ess sup }}

\newcommand{\n}{\|}

\newcommand{\tr}{\text{tr}}

\input xy
\xyoption{all}

\begin{document}
\title[Rearrangement on complex interpolation for families]{On the effect of rearrangement on complex interpolation for families of Banach spaces}
\author{Yanqi QIU}
\date{}
\maketitle
\begin{abstract}
We give a new proof to show that the complex interpolation for families of Banach spaces is not stable under rearrangement of the given family on the boundary, although, by a result due to Coifman, Cwikel, Rochberg, Sagher and Weiss, it is stable when the latter family takes only 2 values. The non-stability for families taking 3 values was first obtained by Cwikel and Janson. Our method links this  problem to the theory of  matrix-valued Toeplitz operator and we are able  to characterize all the transformations on $\T$ that are invariant for complex interpolation at 0, they are precisely the origin-preserving inner functions.
\end{abstract}

\section*{Introduction}This paper is a remark on the theory of complex interpolation for  families of Banach spaces, developed by Coifman, Cwikel, Rochberg, Sagher and Weiss in \cite{CCRSW}. To avoid technical difficulties, we will concentrate on finite dimensional spaces. 

Let $\D = \{z\in \C: | z | < 1\}$ be the unit disc with boundary $\T = \partial \D$. The normalised Lebesgue measure on $\T$ is denoted by $m$. By an interpolation family, we mean a measurable family of complex $N$-dimensional normed spaces $\{E_\gamma: \gamma \in \T\}$, i.e., $E_ \gamma$ is $\C^N$ equipped with norm $\n \cdot\n_ \gamma$ and for each $x \in \C^N$, the function $ \gamma \mapsto \n x\n_ \gamma $ defined on $\T$ is measurable. We should also assume that $\int \log^{+} \n x\n_ \gamma  dm( \gamma ) < \infty$ for any $x \in \C^N$. By definition, the interpolated space at 0 is $$E[0]: = H^{\infty}(\T; \{E_\gamma \})/z H^{\infty}(\T; \{E_\gamma \}).$$ That is, for all  $x \in \C^N$, $$ \n x \n_{E[0]} = \inf\Big \{ \underset{\gamma \in \T}{\esssup} \n f(\gamma)\n_{E_\gamma}  \Big|  f : \T \rightarrow \C^N \text{ analytic}, f(0) = x\Big\}.$$ More generally, for any $z \in \D$, the interpolated space at $z$ for the family $\{ E_\gamma: \gamma \in \T\}$ is denoted by $E[z]$ or $\{ E_\gamma: \gamma \in \T\}[z]$ whose norm is defined as follows. For any $x \in \C^N$,  $$ \n x \n_{E[z]}  := \inf\Big \{ \underset{\gamma \in \T}{\esssup} \n f(\gamma)\n_{E_ \gamma }  \Big| f: \T\rightarrow \C^N \text{ analytic}, f(z) = x\Big\}.$$ It is known (cf. \cite[Prop. 2.4]{CCRSW}) that in the above definition, instead of using $\underset{\gamma \in \T}{\esssup} \n f(\gamma)\n_{E_\gamma}$, we can use $ \Big(\int \n f(\gamma)\n_{E_\gamma}^p \, P_z(d\gamma) \Big)^{1/p}$ for $0 < p < \infty$ or $ \exp\Big(\int \log \n f(\gamma) \n_{E_\gamma}\, P_z(d\gamma)\Big)$  without changing the norm on $E[z]$. Here $P_z(d\gamma)$ is the harmonic measure on $\T$ associated to  $z$.

The goal of this paper is to investigate when the norm of the space $E[0]$ is invariant under a (measure preserving) rearrangement of the family $\{E_\gamma: \gamma \in \T\}$. A trivial example of such a rearrangement is a rotation on $\T$. But, as we will see, there are non trivial instances of this phenomenon. In particular, we recall the following well-known result. 
\begin{thm}\label{known}\cite[Cor. 5.1]{CCRSW}
If $X_\gamma  = Z_0$ for all $\gamma \in \Gamma_0$ and $X_\gamma = Z_1$ for all $\gamma \in \Gamma_1$, where $\Gamma_0$ and $\Gamma_1$ are disjoint measurable sets whose union is $\T$, then $X[0] = (Z_0, Z_1)_\theta$, where $\theta = m(\Gamma_1)$ and $(Z_0, Z_1)_\theta$ is the classical complex interpolation space for the pair $(Z_0, Z_1)$.\end{thm}
The key fact behind this theorem is the existence for any measurable partition $\Gamma_0 \cup \Gamma_1$ of the unit circle of an origin-preserving inner function taking $\Gamma_0$ to an arc of length $2\pi m(\Gamma_0)$ and $\Gamma_1$ to the complementary arc of length.  For details, see the appendix. More generally, complex interpolation at 0 is stable under the rearrangements given by any inner function vanishing at 0.

\begin{prop}\label{invariant-inner}
Let $\varphi: \D \rightarrow \C$ be an inner function vanishing at 0. Its boundary value is denoted again by $ \varphi: \T \rightarrow \T$. Then for any interpolation family $\{E_{\gamma} : \gamma \in \T\}$, the canonical identity: $$Id: \{ E_\gamma: \gamma \in \T\}[0] \rightarrow \{ E_{\varphi(\gamma)}: \gamma \in \T\}[0]$$ is isometric.
\end{prop}

\begin{proof}
The proof is routine, for details, see the last step in the proof of Theorem \ref{known} in the appendix. 
\end{proof}

Theorem \ref{known} shows in particular that in the 2-valued case, the complex interpolation is stable under rearrangement (the reader is referred to Lemma \ref{inner} for the detail). We show that in the general case, this is not the case. We learnt from the referee that  this result was previously obtained by Cwikel and Janson in \cite{Cwikel-Janson} with a different method, the statement is at the bottom of page 214, the proof is from page 278 to page 283. 

Our method is simpler and it also yields a characterization of all the transformations on $\T$ that are invariant for complex interpolation at 0, they are precisely the inner functions vanishing at 0. In other words, the converse of Proposition \ref{invariant-inner} holds.

\bigskip

Here is how the paper is organised. 

 In \S\ref{approximation-formula}, we recall a result from Helson and Lowdenslager's papers \cite{Helson-Low, Helson-Low-II}  on the matrix-valued outer function $F_W: \D \rightarrow M_N$ associated to a given matrix weight $W: \T \rightarrow M_N$. This result allows us to give an approximation formula for $|F_W(0)|^2$ when $W$ is a small perturbation of the constant weight $I$, where $I$ is the identity matrix in $M_N$. 
 
 In \S\ref{continuous-family}, we study the interpolation families consisting of distorted Hilbert spaces (i.e., $\C^N$ equipped with norms $\n x \n_\gamma = \n W(\gamma)^{1/2} x\n_{\ell_2^N}$ for a.e. $\gamma \in \T$). We produce an explicit example of such a family for which complex interpolation at 0 is not stable under rearrangement.  
 
Our main results are given in \S\ref{three-hilbert}, where we study some interpolation families consisting of 3 distorted Hilbert spaces. It is shown that in this restricted case, the complex interpolation at 0 is already non-stable under rearrangement.  One advantage of our method is that we are able to characterize all the transformations on $\T$ that are invariant for complex interpolation at 0, they are precisely the inner functions $\Theta: \T \rightarrow \T$ such that $\Theta(0) = \widehat{\Theta}(0) = 0$.
 
\S\ref{related-comments} is mainly devoted to the stability of complex interpolation under rearrangement for families of \textit{compatible} Banach lattices. We also exhibit a rather surprising non-stability example of interpolation family taking values in  $\{X, \overline{X}, X^*, \overline{X}^*\}$.
 
Finally, in the Appendix, we reformulate the argument of \cite{CCRSW} to prove Theorem \ref{known}, the proof somewhat explains why the 3-valued case is different from the 2-valued case. 


\section{An approximation formula}\label{approximation-formula}
In this section, we first recall some results from \cite[\S 5]{Helson-Low} and \cite[\S10, \S11, \S12]{Helson-Low-II} in the forms that will be convenient for us, and then deduce from them a useful formula.

Let $W: \T \rightarrow M_N$ be a measurable positive semi-definite $N \times N$-matrix valued function such that $\tr(W)$ is integrable. Such a function should be considered as a matrix weight. Without mentioning, all matrix weights in this paper satisfy: There exist $c, C>0$ such that  \begin{eqnarray}\label{assumption-on-weight} c I \le W(\gamma) \le C I \quad \text{ for } a.e.\, \gamma \in \T;\end{eqnarray} where $I$ is the identity matrix in $M_N$.  For such a matrix weight, let $L^2_W= L^2(\T, W; S_2^N)$ be the set of functions $f: \T \rightarrow M_N$ for which $$\n f\n_{L_W^2}^2 = \int \tr \Big(f(\gamma)^* W(\gamma) f(\gamma)\Big) dm(\gamma) < \infty.$$ Clearly, $L_W^2$ is a Hilbert space. 

We will consider two subspaces $H^2(W) \subset L_W^2$ and $H^2_0(W) \subset L_W^2$ defined as follows: $$H^2(W)  = \{ f \in L_W^2 | \hat{f}(n)  = 0, \forall n < 0\}, $$ $$H_0^2(W)  = \{ f \in L_W^2 | \hat{f}(n)  = 0, \forall n \le 0\}.$$ 

Given the assumption \eqref{assumption-on-weight} on $W$, the identity map $Id: L_2(\T; S_2^N) \rightarrow L_W^2$ is an isomorphism, more precisely,  \begin{eqnarray} c^{1/2} \n f\n_{L^2(\T;S_2^N)}  \le  \n f\n_{L^2_W} \le C^{1/2} \n f\n_{L^2(\T;S_2^N)}.\end{eqnarray} In particular, $H^2_0(\T; S_2^N)$ and  $H^2_0(W)$ are set theoretically identical but equipped with equivalent norms. 


In the sequel, any element $F \in H^2(\T; S_2^N)$ will be identified with its holomorphic extension on $\D$, in particular, $F(0) = \widehat{F}(0)$, the 0-th Fourier coefficient.

We recall the following theorem (a restricted form) of Helson and Lowdenslager from \cite{Helson-Low} and \cite{Helson}. We denote by $S_2^N$ the spaces of $N \times N$ complex matrices equipped with the Hilbert-Schmidt norm.

\begin{thm}[Helson-Lowdenslager]
Assume $W$ a matrix weight satisfying the assumption \eqref{assumption-on-weight}. Then there exists $F \in H^2(\T; S_2^N)$ such that \begin{itemize} \item $F(\gamma)^*F(\gamma) = W(\gamma)$ \text{ for } a.e. $\gamma \in \T$. \item $F$ is a right outer function, that is, $F \cdot H^2(\T; S_2^N) $ is dense in $H^2(\T; S_2^N)$. \end{itemize} Let $\Phi$ be the orthogonal projection of the constant function $I$ to the subspace $H^2(W) \ominus H_0^2(W) \subset L_W^2$, i.e., $\Phi = P_{H^2(W) \ominus H_0^2(W)}(I),$ then \begin{eqnarray} \label{outer}\Phi(\gamma)^* W(\gamma) \Phi(\gamma)  = |F(0)|^2 \, \text{ for } a.e.\, \gamma \in \T.\end{eqnarray} Moreover, $\Phi$ and $F$ and both invertible.
\end{thm}

If $F$ and $G$ are two (right) outer functions such that $$F(\gamma)^*F(\gamma) = G(\gamma)^*G(\gamma) = W(\gamma) \, \text{ for } a.e.\, \gamma \in \T,$$ then there is a constant unitary matrix $U \in \mathscr{U}(N)$ such that $F (z) = U G(z)$ for all $z \in \D$. In particular, $| F(0) |^2 = | G(0)|^2$ is uniquely determined by $W$, as shown by the equation \eqref{outer}. Within all possible such outer functions, there is a unique one such that $F(0)$ is positive, we will denote it by $F_W$. 

Let $\Psi = P_{H^2_0(W)}(I)$, where the orthogonal projection $P_{H_0^2(W)}$ is defined on the space $L_W^2$. Clearly, we have \begin{eqnarray}\label{elementaryrel}\Phi = I - \Psi.\end{eqnarray} We have already known that set theoretically, $H^2_0(W) = H^2_0(\T; S_2^N)$, and they are equipped with equivalent norms, thus we have a Fourier series for $\Psi \in H_0^2(W) = H_0^2(\T;S_2^N)$:  $$\Psi = \sum_{n \ge 1} \widehat{\Psi}(n) \gamma^n;$$ where the convergence is in $H^2_0(\T; S_2^N)$ and hence in $H^2_0(W)$.  

By definition, $\Psi$ is characterized as follows. For any $A \in M_N$ and any $n\ge 1$, we have $\langle \Psi, \gamma^n A\rangle_{L_W^2} = \langle I, \gamma^n A\rangle_{L_W^2},$ i.e., $$ \int \tr (\gamma^{-n} A^* W \Psi) dm(\gamma) = \int \tr (\gamma^{-n} A^* W) dm(\gamma).$$ Or equivalently,  \begin{eqnarray}\label{fourierexpan}  \int \gamma^{-n} W \Psi dm(\gamma) = \int \gamma^{-n}  W dm(\gamma), \text{ for } n \ge 1.  \end{eqnarray} We denote by $P_{+}$ the orthogonal projection of $L^2(\T)$ onto the subspace $H^2_0(\T)$. The generalized projection $P_{+} \otimes I_X$ on $L_p(\T; X)$ for $1 < p < \infty$ will still be denoted by $P_{+}$ . Note that $P_{+}$ is slightly different to the usual Riesz projection, the latter is defined as the orthogonal projection onto $H^2(\T)$. Similarly, we denote by $P_{-}$ the orthogonal projection onto $\overline{H_0^2(\T)}$ and also its generalisation on $L_p(\T; X)$ when it is bounded. With this notation, the equation system \eqref{fourierexpan} is equivalent to \begin{eqnarray}\label{riesz} P_{+} (W\Psi) = P_{+}(W).\end{eqnarray}

\bigskip

{\flushleft \bf Key observation:} If $W$ is a perturbation of identity, that is, if there exists a measurable function $\Delta: \T \rightarrow M_N$ such that $$ \Delta(\gamma)^* = \Delta(\gamma) \, \text{ for } a.e.  \gamma \in \T \text{ and } \n \Delta\n_{L_\infty(\T; M_N)} < 1$$ and $$W= I + \Delta;$$ then the equation \eqref{riesz} has the form \begin{eqnarray}\label{perturbation} \Psi  + P_{+}(\Delta \Psi) = P_{+}(\Delta).\end{eqnarray} The above equation can be solved using a Taylor series.

\bigskip

To make the last sentence in the preceding observation rigorous, we introduce the following Toeplitz type operator: $$T_\Delta : H_0^2(\T; S_2^N) \xrightarrow{L_\Delta} L^2(\T; S_2^N) \xrightarrow{P_{+}} H_0^2(\T; S_2^N);$$ where $L_{\Delta}: H_0^2(\T; S_2^N) \rightarrow L^2(\T; S_2^N)$ is the left multiplication by $\Delta$ on the subspace $H_0^2(\T; S_2^N)$. More precisely, $$(L_\Delta f)(\gamma) = \Delta(\gamma) f(\gamma) \text{ for any } f \in L^2(\T; S_2^N).$$ Clearly, we have $$ \n T_\Delta\n \le \n \Delta\n_{L_\infty(\T; M_N)}< 1.$$ The term $P_{+}(\Delta)$ in equation \eqref{perturbation} should be treated as an element in $H_0^2(\T; S_2^N)$, then the equation \eqref{perturbation} has the form \begin{eqnarray}\label{toeplitz} (Id + T_{\Delta}) (\Psi) = P_{+}(\Delta).\end{eqnarray} Since $\n T_\Delta\n < 1$, the operator $Id + T_\Delta$ is invertible. Thus equation \eqref{toeplitz} has a unique solution $\Psi \in H_0^2(\T; S_2^N) = H^2_0(W)$ given by the formula: \begin{eqnarray}\label{taylorseries} \Psi & = & (Id + T_\Delta)^{-1} (P_{+}(\Delta))  =  \sum_{n = 0 }^\infty(-1)^n T_\Delta^n(P_{+}(\Delta)) ;\end{eqnarray} where $T^0_\Delta (P_{+}(\Delta)) = P_{+}(\Delta)$, and the convergence is understood in the space $H^2_0(\T; S_2^N)$.  Combining equations \eqref{outer}, \eqref{elementaryrel} and \eqref{taylorseries}, we deduce the following formula: \begin{align*}  |F_{I + \Delta}(0) |^2 = &\Big[I - \sum_{n = 0}^\infty(-1)^n T_{\Delta}^n(P_{+}(\Delta))\Big]^*(I+ \Delta) \times \\ & \times \Big[I - \sum_{n= 0}^\infty(-1)^n T_{\Delta}^n(P_{+}(\Delta))\Big] .\end{align*}

We summarize the above discussion in the following:
\begin{prop}\label{degree2}
Let $\Delta: \T \rightarrow M_N$ be a measurable bounded selfadjoint function such that $\n \Delta \n_{L_\infty(\T; M_N)} < 1$. Let $\varepsilon \in [0, 1]$, then we have \begin{align*} |F_{I + \varepsilon \Delta}(0) |^2  = &\Big[I - \sum_{n = 0}^\infty(-1)^n \varepsilon^{n+1} T_{\Delta}^n(P_{+}(\Delta))\Big]^*(I+\varepsilon \Delta) \times \\ & \times \Big[I - \sum_{n= 0}^\infty(-1)^n \varepsilon^{n+1}T_{\Delta}^n(P_{+}(\Delta))\Big]. \end{align*} In particular, we have \begin{align}\label{approximation} | F_{I  + \varepsilon \Delta}(0) |^2  = I + \varepsilon \widehat{\Delta}(0) - \varepsilon^2 \sum_{n \ge 1} |\widehat{\Delta}(n) |^2 + \mathcal{O}(\varepsilon^3), \text{ as } \varepsilon \to 0^{+}. \end{align}
\end{prop}
\begin{proof}
It suffices to prove the approximation identity \eqref{approximation}. We have \begin{align}\label{expansion} \begin{split} | F_{I + \varepsilon \Delta} (0) |^2  = & \Big[ I - \varepsilon P_{+} (\Delta) + \varepsilon^2 T_{\Delta}(P_{+}(\Delta)) + \mathcal{O}(\varepsilon^3) \Big]^* (I + \varepsilon \Delta) \times  \\  & \times  \Big[ I - \varepsilon P_{+} (\Delta) + \varepsilon^2 T_{\Delta}(P_{+}(\Delta)) + \mathcal{O}(\varepsilon^3) \Big]  \\  = &  I + \varepsilon  R_1 + \varepsilon^2 R_2 + \mathcal{O}(\varepsilon^3), \text{ as } \varepsilon \to 0^{+}; \end{split} \end{align} where $$R_1 = \Delta -  P_{+}(\Delta) - P_{+}(\Delta)^*,$$ $$ R_2 = P_{+}(\Delta)^*P_{+}(\Delta)   -  \Delta P_{+}(\Delta) -  P_{+}(\Delta)^* \Delta +   T_\Delta(P_{+}(\Delta)) +  T_{\Delta}(P_{+}(\Delta))^*.$$ For $R_1$, we note that since $\Delta$ is selfadjoint, $ P_{-}(\Delta)= P_{+}(\Delta)^*$ and hence \begin{eqnarray}\label{selfadjoint} \Delta = P_{+}(\Delta) + P_{+}(\Delta)^* + \widehat{\Delta}(0).\end{eqnarray} Thus $$ R_1 = \widehat{\Delta}(0). $$ For $R_2$, we note that since the left hand side of equation \eqref{expansion} is independent of $\gamma \in \T$, the right hand side should also be independent of $\gamma$, hence $R_2$ must be independent of $\gamma$, it follows that \begin{eqnarray*}R_2 &= & \int R_2 (\gamma) dm(\gamma)  \\ & = &  \int \Big( P_{+}(\Delta)^*P_{+}(\Delta)   -  \Delta P_{+}(\Delta) -  P_{+}(\Delta)^* \Delta\Big) dm(\gamma) \\ & = & - \sum_{n \ge 1}\widehat{ \Delta}(n)^*\widehat{\Delta}(n) = - \sum_{n \ge 1} | \widehat{\Delta}(n)|^2.\end{eqnarray*} 
\end{proof}


\section{Interpolation Families in  the Continuous Case}\label{continuous-family}
To any invertible matrix $A\in GL_N(\C)$ is associated a Hilbertian norm $\n \cdot\n_A$ on $\C^N$, which is defined as follows: $$\n x \n_A = \n A x\n_{\ell_2^N}, \text{ for  any } x \in \C^N;$$  where $\ell_2^N$ denotes the space $\C^N$ with the usual Euclidean norm.  Let us denote  $\ell_A^2 :  = (\C^N, \n \cdot \n_A).$ We have the following elementary properties: 
 
 \begin{itemize}
 \item Let $A, B \in GL_N(\C)$, then they define the same norm on $\C^N$ if and only if $| A | = | B|$.  Thus, if $U\in \mathscr{U}(N)$ is a $N \times N$ unitary matrix, then $\n \cdot\n_{UA} = \n \cdot\n_A$. 
 
 \item We define a pairing  $(x, y)= \sum_{n = 1}^N x_n y_n$ for any $x, y \in \C^N$, then under this pairing, we have the  canonical isometries: $$(\ell_A^2)^* = \ell_{A^{-T}}^2;$$  where  $A^{-T}$ is the inverse of the tranpose matrix $A^T$.
 
 \item We have the following canonical isometries: $$\overline{\ell_A^2} = \ell_{\overline{A}}^2 \,\text{ and }\, \overline{\ell_A^2}^* =  \ell_{(A^*)^{-1}}^2.$$
 
 \end{itemize}
Here we recall that, for a complex Banach space $X$, its complex conjugate $\overline{X}$ is defined to be the space consists of the same element of $X$, but with scalar multiplication $$\lambda \cdot v= \bar{\lambda} v, \text{ for } \lambda \in \C, v \in X. $$
\bigskip

Consider an $N\times N$-matrix weight $W$. To such a weight is associated  an interpolation family $$\{ \ell^2_{w(\gamma)}: \gamma \in \T\}, \text{ where }w(\gamma) = \sqrt{W(\gamma)}.$$ 

The following elementary proposition will be used frequently:
\begin{prop}\label{startpoint}
For interpolation family $\{E_\gamma: \gamma \in \T\}$ with $E_\gamma = \ell_{w(\gamma)}^2$, we have $E[0] = \ell^2_{F(0)}$, that is, $$\n x \n_{E[0]} = \n F(0) x \n_{\ell_2^N}, \, \text{ for all } x \in \C^N;$$ where $F(z)$ is any right outer function associated to the weight $W$.
\end{prop}
\begin{proof}
By the definition of right outer function associated to the weight $W$, \begin{eqnarray}\label{defrightouter} F(\gamma)^* F(\gamma) = W(\gamma)\,\text{ for } a.e.\, \gamma \in \T.\end{eqnarray} For any $x \in \C^N$,  define an analytic function $f_x: \D \rightarrow \C^N$ by $$f_x(z) = F(z)^{-1} F(0) x,$$ then $f_x(0) = x$ and for a.e. $\gamma \in \T$, \begin{eqnarray*} \n f_x(\gamma)\n_{w(\gamma)}^2 & = & \langle W(\gamma) F(\gamma)^{-1} F(0) x, F(\gamma)^{-1} F(0) x\rangle \\ & = &\langle F(\gamma)^* F(\gamma) F(\gamma)^{-1} F(0) x, F(\gamma)^{-1} F(0) x\rangle \\ & = & \n F(0) x  \n_{\ell_2^N}^2.\end{eqnarray*} This shows that $\n f_x\n_{H^{\infty}(\T; \{E_\gamma\})} \le \n F(0) x\n_{\ell_2^N}$, whence $$\n x\n_{E[0]} \le \n F(0) x \n_{\ell_2^N} = \n x \n_{\ell_{F(0)}^2}.$$

 The converse inequality will be given by duality, it suffices to show that $$\n x \n_{E[0]^*}  \le \n x \n_{\ell_{F(0)^{-T}}^2} = \n x \n_{(\ell_{F(0)}^2)^*}.$$ Consider the dual interpolation family $\{E_\gamma^*: \gamma \in \T\} = \{\ell_{w(\gamma)^{-T}}^2: \gamma \in \T\},$ which is naturally given by the weight $W(\gamma)^{-T} = (w(\gamma)^{-T})^*w(\gamma)^{-T}$. By \cite[Th. 2.12]{CCRSW}, we have a canonical isometry $$\{E_\gamma^*: \gamma \in \T\} [0] =  E[0]^*.$$ The identity \eqref{defrightouter} implies $$(F(\gamma) ^{-T})^* F(\gamma) ^{-T} = W(\gamma)^{-T}\,\text{ for } a.e.\, \gamma \in \T. $$ Thus $F(z)^{-T}$ is the right outer function associated to the weight $W(\gamma)^{-T}$. Then the same argument as above yields that $$\n x \n_{E[0]^*}  \le \n x \n_{\ell_{F(0)^{-T}}^2} = \n x \n_{(\ell_{F(0)}^2)^*}.$$
\end{proof}

\begin{rem}
More generally, assume that $X$ is a (finite dimensional) normed space such that $M_N \subset End(X)$ and $\n u \cdot x \n_X = \n x\n_X$ for any $u \in \mathscr{U}(N)$. For instance $X= S_p^N \, (1\le p \le \infty)$ and $M_N$ acts on $S_p^N$ by the usual left multiplications of matrices. Consider the interpolation family $E_\gamma = (X, \n \cdot\n_{X; \, A(\gamma)})$ with $\n x \n_{X;\,  A(\gamma)} = \n A(\gamma) \cdot x \n_X$ for any $\gamma \in \T$, then $E[0] = (X, \n \cdot \n_{B(0)})$ with $\n x \n_{B(0)} = \n B(0) \cdot x \n_X$, where $B(z)$ is any right outer function associated to the matrix weight $A(\gamma)^*A(\gamma)$.  
\end{rem}

The following result is probably known to the experts of prediction theory, since we do not find it in the literature, we include its proof. 
\begin{prop}\label{basicprop}
The function  $\{W(\gamma): \gamma \in \T\} \mapsto F_W(0)$ or equivalently $\{W(\gamma): \gamma \in \T\} \mapsto | F(0)|^2$ is not stable under rearrangement. More precisely, there exists a family $\{W(\gamma): \gamma \in \T\}$ and a measure preserving mapping $S : \T\rightarrow \T$, such that $$F_W(0) \ne F_{W\circ S }(0).$$
\end{prop}
Before we proceed to the proof of the proposition, let us mention that if the weight $W(\gamma)$ takes only 2 distinct values, i.e., if $W(\gamma) = A_0$ for $\gamma \in \Gamma_0$  and $W(\gamma) = A_1$ for $\gamma \in \Gamma_1$ with $\T = \Gamma_0 \cup \Gamma_1$ a measurable partition, then a detailed computation shows that we have $$F_W(0)^2 = A_0^{1/2} (A_0^{-1/2} A_1 A_0^{-1/2})^{m(\Gamma_1)} A_0^{1/2} = A_1^{1/2} (A_1^{-1/2} A_0 A_1^{-1/2})^{m(\Gamma_0)} A_1^{1/2} .$$ In particular,  $F_W(0) = F_{W \circ M}(0)$ for any measure preserving mapping $M: \T\rightarrow \T$. Of course, this can be viewed as a special case of Theorem \ref{known}. The fact that we can calculate $F_W(0)$ efficiently in the above situation is due to the fundamental fact that two quadratic forms can always be simultaneously diagonalized. 
\begin{proof}
Fix $r > 0$, define two $M_2$-valued bounded analytic functions $F_1, F_2: \D \rightarrow M_2$ by  $$F_1(z) = \left[\begin{array}{cc} (1+r^2)^{1/4} & r(1+r^2)^{-1/4} z \\ 0 & (1+r^2)^{-1/4} \end{array}\right],$$ $$F_2(z) = \left[\begin{array}{cc} (1+r^2)^{-1/4} & 0 \\ r(1+r^2)^{-1/4} z & (1+r^2)^{1/4} \end{array}\right].$$ Note that they are both outer since $z \rightarrow F_1(z)^{-1}$ and $z \rightarrow F_2(z)^{-1}$ are bounded on $\D$. By a direct computation, $$F_1(e^{i \theta} )^* F_1(e^{i \theta}) = W_1 (e^{i \theta}) = \left[\begin{array}{cc} (1+ r ^2 )^{1/2} & r e^{i \theta} \\ r e^{-i \theta} & (1+ r^2)^{1/2} \end{array} \right],$$  $$F_2(e^{i \theta} )^* F_2(e^{i \theta}) = W_2 (e^{i \theta}) = \left[\begin{array}{cc} (1+ r ^2 )^{1/2} & r e^{-i \theta} \\ r e^{i \theta} & (1+ r^2)^{1/2} \end{array} \right].$$ If we define $S: \T \rightarrow \T$ by $S(\gamma) = \overline{\gamma}$, then $S$ is measure preserving and $W_2 = W_1 \circ S$.  By noting that $F_1(0)$ and $F_2(0)$ are positive, we have $F_1= F_{W_1}$ and $F_2 = F_{W_2} = F_{W_1 \circ S}$. However, $F_{W_1 \circ S}(0) = F_2(0) \ne F_1(0) =  F_{W_1}(0)$.
\end{proof}

We denote $$W^{(r)}(e^{i\theta}): =  \left[\begin{array}{cc} (1+ r ^2 )^{1/2} & r e^{i \theta} \\ r  e^{-i \theta}& (1+ r^2)^{1/2} \end{array} \right],$$ and let $w^{(r)}(\gamma) = \sqrt{W^{(r)}(\gamma)}$. The notation $S: \T \rightarrow \T$ will be reserved for the complex conjugation mapping.

An immediate consequence of Propositions \ref{startpoint} and \ref{basicprop} is the following: 
\begin{cor}\label{continuous}
The interpolation family $\{ \widetilde{E}^{(r)}_\gamma = \ell_{(w^{(r)} \circ S)(\gamma)}^2 : \gamma \in \T\}$ is a rearrangement of the family $\{ E^{(r)}_\gamma = \ell^2_{w^{(r)}(\gamma)}: \gamma \in \T\}$. The identity mapping $ Id: \widetilde{E}^{(r)} [0] \rightarrow E^{(r)} [0]$  has norm $$\n Id: \widetilde{E}^{(r)} [0] \rightarrow E^{(r)} [0]\n  = (1+r^2)^{1/2}.$$
\end{cor}
\begin{proof}
Indeed, we have: \begin{align*}  &  \n Id: \widetilde{E}^{(r)} [0] \rightarrow E^{(r)} [0]\n = \sup_{ x \ne 0} \frac{\n F_{W^{(r)}} (0) x \n_{\ell_2^2}}{\n F_{W^{(r)} \circ S }(0)x \n_{\ell_2^2}} \\ & =  \n F_{W^{(r)}}(0) F_{W^{(r)} \circ S} (0)^{-1}  \n_{M_2}  =  (1+r^2)^{1/2}.\end{align*}
\end{proof}

\begin{rem}

By Corollary \ref{continuous} and a suitable discretization argument, we can show that if  $J_k = \Big\{e^{i \theta}: \frac{(k-1)\pi}{4} \le \theta < \frac{k \pi}{4} \Big\},$ for $ 1 \le k \le 8$, and let $\gamma_k \in J_k $ be the center point on $J_k$, then the interpolation families $B^{(r_0)}_\gamma = \ell^2_{w^{(r_0)}(\gamma_k)}$ if $\gamma \in J_k$ and $\widetilde{B}^{(r_0)}_\gamma = \ell^2_{w^{(r_0)}(\bar{\gamma}_k)}$ if $ \gamma \in J_k$ for $r_0 = \sqrt{2+2\sqrt{2}}$ give different interpolation space at 0, i.e., $$\n Id: \widetilde{B}^{(r_0)} [0] \rightarrow B^{(r_0)}[0]\n > 1.$$  We omit its proof, because in the next section, we give a better result by using the formula obtained in \S \ref{approximation-formula}. \end{rem}


\section{Interpolation for three Hilbert spaces}\label{three-hilbert}
In this section, we will show that complex interpolation is not stable even for a familiy taking only 3 distinct Hilbertian spaces. The starting point of this section is Proposition \ref{approximation-formula}. Our proof is somewhat abstract, but it explains why the  3-valued case becomes different from the 2-valued case, the idea used in the proof will be applied further to get a characterization of measurable transformations on $\T$ that perserve complex interpolation at 0. 

\begin{thm}\label{abstract-proof}
There are two different measurable partitions of the unit circle: $$\T  = S_1 \cup S_2 \cup S_3 =  S_1' \cup S_2' \cup S_3', \,\, m(S_k) = m(S_k'), \text{ for } k  =1, 2, 3,$$ and three constant selfadjoint matrices  $\Delta_k \in M_2$ for $k = 1, 2, 3$, such that if we let $$ \Delta = \Delta_1 1_{S_1} + \Delta_2 1_{S_2}  + \Delta_3 1_{S_3} \text{ and } \Delta' = \Delta_1 1_{S_1'} + \Delta_2 1_{S_2'}  + \Delta_3 1_{S_3'},$$ then $$ \sum_{n \ge 1} | \widehat{\Delta}(n) |^2 \ne  \sum_{n \ge 1} | \widehat{\Delta'}(n) |^2.$$
\end{thm}

Before turning to the proof of the above theorem, we state our main result. 
 \begin{cor}\label{maincor}
Let $\Delta, \Delta'$ be as in Theorem \ref{abstract-proof}. For $0 < \varepsilon<  \frac{1}{\,\, \n \Delta\n_\infty}$, we define two matrix weights which are perturbation of identity: $$ W_\varepsilon = I + \varepsilon \Delta, \, \, W_\varepsilon' = I + \varepsilon \Delta' .$$ Denote $w_\varepsilon$ and $w_{\varepsilon}'$ the square root of $W_\varepsilon$ and $W'_\varepsilon$ respectively. Then there exists $\varepsilon_0< 1$ such that whenever $0 < \varepsilon < \varepsilon_0$, we have $$| F_{W_{\varepsilon}} (0) |^2 \ne | F_{W_{\varepsilon}'} (0) |^2. $$ Thus, whenever $0 < \varepsilon < \varepsilon_0$,  the following two interpolation families $$\{ \ell^2_{w_{\varepsilon}(\gamma)}: \gamma \in \T\}, \quad \{ \ell^2_{w'_{\varepsilon}(\gamma)}: \gamma \in \T\}$$ have the same distribution and take only 3 distinct values. However, the interpolation spaces at 0 given by these two families are different: $$\ell^2_{F_{W_\varepsilon(0)}} \ne \ell^2_{F_{W'_\varepsilon(0)}} .$$ 
\end{cor}

\begin{proof}
This  is an immediate  corollary of Proposition \ref{degree2} and Theorem \ref{abstract-proof}. The last assertion follows from Proposition \ref{startpoint}.
\end{proof}
 
\begin{rem}
We verify that in the two main cases where the interpolation is stable under rearrangement, the function $\Delta \mapsto \sum_{n \ge 1} | \hat{\Delta}(n)|^2$ is stable under rearrangement. Note first that we have the following matrix identity: $$ \sum_{n \ge 1} | \widehat{\Delta}(n)|^2 =  \int |P_{+}(\Delta)|^2 dm.$$ 
\begin{itemize} 
\item  2-valued case: If $\Delta$ is a 2-valued selfadjoint function, i.e, there is a measurable subset $A\subset \T$ and two selfadjoint matrices  $\Delta_1, \Delta_2 \in M_N$, such that $\Delta = \Delta_1 1_A+ \Delta_2 1_{A^{c}}$ then \begin{align*} \sum_{n \ge 1} | \widehat{\Delta}(n)|^2 =&  \int |P_{+}(\Delta)|^2 dm   = \int \Big|P_{+} \Big((\Delta_1 - \Delta_2)1_A  + \Delta_2 \Big)\Big|^2 dm \\  = & | \Delta_1-\Delta_2|^2 \int |P_{+} (1_A)|^2 dm \\ =&  \frac{m(A)- m(A)^2}{2}  | \Delta_1-\Delta_2|^2 ,\end{align*} which depends on the measure of $A$ but not the other structure of $A$. 

More generally, we note in passing that for any real valued $f$ in $L_2(\T)$ the expression $2 \n P_{+}(f)\n_2^2 = 2 \sum_{n \ge 1} | \widehat{f}(n)|^2$ coincides with the variance of $f$.

\item Rearrangement under inner functions: Let $\varphi: \T\rightarrow\T$ be the boundary value of an origin-preserving inner function. Assume $\Delta: \T \rightarrow M_N$ selfadjoint. Note that $P_{+}(\Delta \circ \varphi)  = P_{+}(\Delta) \circ \varphi$ and that $\varphi$ preserves the measure $m$. Hence \begin{align*}\sum_{n\ge 1} | \widehat{(\Delta \circ \varphi)}(n)|^2 = & \int | P_{+}(\Delta\circ\varphi) |^2  dm =\int | P_{+}(\Delta)\circ\varphi |^2  dm  \\  = & \int | P_{+}(\Delta) |^2  dm =  \sum_{n \ge 1} | \widehat{\Delta}(n)|^2 .\end{align*}
\end{itemize} 

\end{rem}

\begin{proof}[Proof of Theorem \ref{abstract-proof}] 
Assume by contradiction that for any pair of 3-valued selfadjoint functions $\Delta$ and $\Delta'$ as in the statement of Theorem \ref{abstract-proof}, we have \begin{eqnarray}\label{contra-assumption}\sum_{n \ge 1} | \widehat{\Delta}(n)|^2 = \sum_{n \ge 1} | \widehat{\Delta'}(n)|^2.\end{eqnarray} We make the following reduction.

{\flushleft \bf Step 1:} The above assumption implies that for any pair of functions, $\Delta, \Delta' $ taking values in the same set of three matrices and having identical distribution, the equation \eqref{contra-assumption} holds as well. Indeed, given such a pair, we can consider the  pair of selfadjoint functions which are still 3-valued:  $$\gamma \rightarrow \left[ \begin{array}{cc} 0 & \Delta(\gamma)^* \\ \Delta(\gamma) & 0 \end{array}\right] \text{ and } \gamma \rightarrow \left[ \begin{array}{cc} 0 & \Delta'(\gamma)^* \\ \Delta'(\gamma) & 0 \end{array}\right].$$ Then the square of the $n$-th Fourier coefficient becomes $$\left[\begin{array}{cc} | \widehat{\Delta}(n)|^2 & 0 \\ 0 & | \widehat{\Delta^*}(n)|^2 \end{array}\right] \text{ and } \left[\begin{array}{cc} | \widehat{\Delta'}(n)|^2 & 0 \\ 0 & | \widehat{\Delta'^*}(n)|^2 \end{array}\right]   $$ respectively. The block (1, 1)-terms then give the desired equation. 

{\flushleft \bf Step 2:} If we take $N = 1$ in the above step, then the conclusion is that for any pair of 3-valued scalar functions $f, f'\in L_\infty(\T)$ such that $f \stackrel{d}{=} f'$, we have $ \sum_{n \ge 1} | \widehat{ f} (n)|^2 = \sum_{n \ge 1}| \widehat{f'}(n)|^2$, or  equivalently, $$\n P_{+}(f) \n_2^2 = \n P_{+}(f')\n_2^2.$$ 

{\flushleft \bf Consequence I:} Under the above assumption, if  $(A_1, A_2)$ is a pair of two disjoint measurable subsets of $\T$, and $(A_1', A_2')$ is another  such pair such that $m(A_1) = m(A_1')$ and $m(A_2) = m(A_2')$, then \begin{eqnarray}\label{claim}\langle P_{+}(1_{A_1}), P_{+}(1_{A_2}) \rangle_{L^2(\T)} = \langle P_{+}(1_{A'_1}), P_{+}(1_{A'_2}) \rangle_{L^2(\T)}  \end{eqnarray} Indeed, if we define $A_3: = \T \setminus (A_1 \cup A_2)$ and  $A'_3: = \T \setminus (A'_1 \cup A'_2)$. For any $\alpha \in \C, \alpha \ne 0 , 1$, consider $$f_\alpha = \alpha 1_{A_1} + 1_{A_2} + 0 \times 1_{A_3}, \quad f'_\alpha = \alpha 1_{A'_1} + 1_{A'_2} + 0 \times 1_{A'_3},$$ then $f_\alpha $ and $f'_\alpha$ are two functions taking exactly 3 values 0, 1, $\alpha$ and $f_\alpha \stackrel{d}{=} f'_\alpha$. Hence by the assumption, we have \begin{align}\label{alpha-value} \n \alpha P_{+}(1_{A_1}) + P_{+}(1_{A_2})\n_2^2 = \n \alpha P_{+}(1_{A'_1}) + P_{+}(1_{A'_2})\n_2^2, \text{ for any } \alpha \in \C.\end{align} Note that for any measurable set $A$, since $1_A$ is real, \begin{eqnarray}\label{real} \n P_{+}(1_A)\n_2^2  = \frac{m(A)-m(A)^2}{2}\end{eqnarray} Taking this in consideration, the equation \eqref{alpha-value} implies that $$\Re\Big(\alpha\langle P_{+}(1_{A_1}), P_{+}(1_{A_2}) \rangle\Big) = \Re\Big(\alpha\langle P_{+}(1_{A'_1}), P_{+}(1_{A'_2}) \rangle\Big), \text{ for any } \alpha \in \C, $$ hence the equation \eqref{claim} holds.  

{\flushleft \bf Step 3:} We can deduce from our assumption the following consequence. 
{\flushleft \bf Consequence II:} For any pair of scalar functions  $f, f'\in L_\infty(\T)$  (without the assumption that they are both 3-valued), such that $f \stackrel{d}{=} f'$ , we have $\n P_{+}(f) \n_2 = \n P_{+}(f')\n_2.$ 

Indeed, if $$ f= \sum_{k  =1}^n f_k 1_{A_k}, \quad f' = \sum_{k  =1}^n f_k 1_{A'_k}, $$ where $(A_k)_{k=1}^n$ are disjoint subsets of $\T$, so is $(A'_k)_{k=1}^n$, moreover $m(A_k) = m(A_k')$. By \eqref{claim} and \eqref{real}, we have \begin{eqnarray*} \n P_{+}(f) \n_2^2 & = & \sum_{k =1}^n | f_k|^2 \cdot \n P_{+}(1_{A_k})\n_2^2 + \sum_{1\le k  \ne l \le n } f_k f_l  \langle P_{+}(1_{A_1}), P_{+}(1_{A_2}) \rangle \\ &  = &  \sum_{k =1}^n | f_k|^2 \cdot \n P_{+}(1_{A'_k})\n_2^2 + \sum_{1\le k  \ne l \le n } f_k f_l  \langle P_{+}(1_{A'_1}), P_{+}(1_{A'_2}) \rangle \\ & = &   \n P_{+}(f') \n_2^2.\end{eqnarray*} Then by an approximation argument,  more precisely, by using the fact that two functions $f, f' \in L^2(\T)$ such that $f \stackrel{d}{=} f$  can be approximated in $L^2(\T)$ by two sequences of simple functions $(g_n)$ and $(g_n')$ such that $g_n \stackrel{d}{=} g_n'$, we can extend the above equality for pairs of equidistributed simple functions to the general equidistributed pairs of functions, as stated in Consequence II.

{\flushleft \bf Step 4:} Now if we take $f, f' \in L_\infty(\T)$ to be  $f(\gamma) = \gamma$ and $f'(\gamma)= \overline{\gamma}$, then $f\stackrel{d}{=} f'$, but we have $ \n P_{+}(f)\n_2 = 1 \ne 0 = \n P_{+}(f')\n_2,$ which contradicts Consequence II. This completes the proof.

\end{proof}

Define $$T_k: = \Big\{ e^{i \theta}| \frac{2 (k-1)\pi}{3} \le \theta < \frac{2k\pi}{3}\Big\} \text{ for } k = 1, 2, 3.$$ We claim that in Theorem \ref{abstract-proof} and hence in Corollary \ref{maincor}, we can take for example $$S_1 = S_1' = T_1, \quad S_2 = S_3'= T_2,\quad S_3 = S_2' = T_3.$$ Indeed, by the proof of Theorem \ref{abstract-proof}, here  we only need to show that \begin{eqnarray*}\langle P_{+} (1_{T_1}), P_{+}(1_{T_2} )\rangle \ne  \langle P_{+} (1_{T_1}), P_{+}(1_{T_3} )\rangle. \end{eqnarray*} Since $1_{T_1}(\gamma)= 1_{T_3} (e^{- i\frac{2 \pi}{3}} \gamma)$ and  $1_{T_2}(\gamma)= 1_{T_1} (e^{- i\frac{2 \pi}{3}} \gamma)$, we have $$\langle P_{+} (1_{T_1}), P_{+}(1_{T_3} )\rangle = \langle P_{+} (1_{T_2}), P_{+}(1_{T_1} )\rangle.$$ Thus we only need to show that \begin{eqnarray}\label{three-arc}  \Im \Big(\langle P_{+} (1_{T_1}), P_{+}(1_{T_2} )\rangle  \Big) \ne 0.\end{eqnarray} Note that $$ \Im \Big(\widehat{1_{T_1}}(n) \overline{\widehat{1_{T_2}}(n)}\Big)= \frac{\sin\frac{2\pi}{3}(1-\cos\frac{2\pi}{3})}{2\pi^2 n^2} \times \left \{ \begin{array}{cl} 0, &  \text{ if } n \equiv 0\mod3;\\ 1, & \text{ if } n \equiv 1 \mod3;\\ - 1, &  \text{ if } n\equiv 2 \mod 3.\end{array} \right.$$
Hence \begin{eqnarray*} \Im \Big(\langle P_{+} (1_{T_1}), P_{+}(1_{T_2} )\rangle  \Big)  & = &  \frac{3 \sin\frac{2\pi}{3}(1-\cos\frac{2\pi}{3})}{2\pi^2}\sum_{k = 0}^\infty \frac{2k + 1}{(3k+1)^2(3k+2)^2},  \end{eqnarray*} which is non-zero, as we expected.

The same idea as in the proof of Theorem \ref{abstract-proof} yields the following characterization:  combining with Proposition \ref{invariant-inner}, we have characterized all measurable transformations on $\T$ that preserve complex interpolation at 0. At this stage, the proof is quite direct.

\begin{thm}\label{car}
Let $\Theta: \T \rightarrow \T$ be a measurable transformation. If for any interpolation family $\{E_\gamma; \gamma \in \T\}$, we have $$\{E_\gamma: \gamma \in \T\}[0] = \{ E_{\Theta(\gamma)}: \gamma\in \T\}[0],$$ then $\Theta$ is an inner function and $\hat{\Theta}(0) = 0$. 
\end{thm}

\begin{rem}
The main point of Theorem \ref{car} is to characterize all the transformations which preserve the interpolation spaces at origin.
\end{rem}

\begin{proof}
 It suffices to show that $\Theta \in H_0^\infty(\T)$, since by definition $\Theta(\gamma)$ has modulus 1 for $a.e. \,\gamma\in \T$. By Propositions \ref{degree2}, \ref{startpoint} and similar arguments in the proof of Theorem \ref{abstract-proof}, we have \begin{eqnarray}\label{L2normeq} \n P_{+}(f \circ \Theta) \n_2 = \n P_{+} (f) \n_2, \text{ for any scalar function } f \in L_\infty(\T). \end{eqnarray} Now take $f(\gamma) = \overline{\gamma}$, we have $\n P_{+} (\overline{\Theta})\n_2 = \n P_{+}(\overline{\gamma})\n_2 =  0$, which implies that $\overline{\Theta} \in 
 \overline{H^\infty(\T)}$ and hence $\Theta \in H^\infty(\T)$. Then we can write $\Theta = \widehat{\Theta}(0) + P_{+}(\Theta)$. In \eqref{L2normeq}, if we take $f(\gamma)= \gamma$, then $ \n P_{+}(\Theta)\n_2= \n P_{+}(\gamma)\n_2 = 1$. Note that $$1 = \n \Theta \n_2^2 = | \widehat{\Theta}(0)|^2 + \n P_{+}(\Theta)\n_2^2,$$ whence $\widehat{\Theta}(0) = 0$. This completes the proof. 
 \end{proof}

\section{Some related comments}\label{related-comments}
Recall that an $N$-dimensional (complex) Banach space $\mathscr{L}$ is called a (complex) Banach lattice with respect to a fixed basis $(e_1, \cdots, e_N)$ of $\mathscr{L}$ if it satisfies the lattice axiom: For any $x_k, y_k\in \C$ such that $| x_k| \le |y_k|$ for all $1\le k \le N$, $$\n \sum_{k = 1}^N x_k e_k \n_{\mathscr{L}} \le \n \sum_{k= 1}^N y_k e_k \n_{\mathscr{L}}.$$ Thus in particular, $$\n \sum_{k = 1}^N x_k e_k \n_{\mathscr{L}} = \n \sum_{k= 1}^N |x_k| e_k \n_{\mathscr{L}}.$$ The above fixed basis $(e_1, \cdots, e_N)$ will be called a lattice-basis of $\mathscr{L}$. Such a Banach lattice $\mathscr{L}$ will be viewed as function spaces over the $N$-point set $[N] = \{1, \cdots, N\}$ in such a way that $e_k$ corresponds to the Dirac function at the point $k$. Thus for $x, y \in \mathscr{L}$, we can write $| x | \le | y |$ if $ |x_k| \le | y_k|$ for all $ 1 \le k \le N$, and $\log | x | = \sum_{k=1}^N \log | x_k| e_k$, suppose that $x_k \ne 0$ for all $1\le k \le N$.

We will call $\{\mathscr{L}_\gamma = (\C^N, \n \cdot\n_\gamma): \gamma \in \T\}$ a family of \textit{compatible} Banach lattices, if there is an algebraic basis $(e_1, \cdots, e_N)$ of $\C^N$ which is simultaneously a lattice-basis of $\mathscr{L}_\gamma$ for a.e. $\gamma \in \T$ and such that \begin{eqnarray}\label{compatible-lattice} 0 <  \underset{\gamma \in \T}{\text{ess inf }} \n e_k \n_\gamma \le  \underset{\gamma\in \T}{\esssup} \n e_k \n_\gamma < \infty  \text{ for all }   1 \le k \le N. \end{eqnarray}


In the sequel, the notation $\{\mathscr{L}_\gamma = (\C^N, \n \cdot\n_\gamma): \gamma \in \T\}$ is reserved for a family of compatible Banach lattices with respect to the \textit{canonical basis} of $\C^N$.

Complex interpolation at 0 for families of compatible Banach lattices is stable under any rearrangement. The proof of the following proposition is standard.

\begin{prop}\label{proplattice}
If $\{ \mathscr{L}_\gamma = (\C^N, \n \cdot\n_\gamma): \gamma \in \T\}$ be an interpolation family of compatible Banach lattices, then \begin{eqnarray}\label{description-lattice}\log \n x\n_{\mathscr{L}[0]} = \inf  \int \log \n f(\gamma)\n_\gamma \, dm(\gamma),\end{eqnarray} where the infimum runs over the set of all  measurable coordinate bounded functions $f: \T \rightarrow \C^N$, i.e., $f_k: \T \rightarrow \C$  is bounded for all $1\le k \le N$  such that $($ by convention $\log 0 : = - \infty$ $)$ $$\log |x| \le \int \log | f(\gamma) | \, dm(\gamma).$$

In particular, if $M: \T\rightarrow \T$ is measure preserving and let $\{ \widetilde{\mathscr{L}}_\gamma= \mathscr{L}_{M(\gamma)}: \gamma \in \T\}$, then $$Id: \mathscr{L}[0] \rightarrow \widetilde{\mathscr{L}}[0]$$ is isometric.
\end{prop} 
\begin{proof}
It suffices to show \eqref{description-lattice}. Assume that $x \in \C^N$ and $\n x\n_{\mathscr{L}[0]}< \lambda$. Without loss of generality, we can assume $x_k \ne 0$ for all $1\le k \le N$.   By the definition of $\mathscr{L}[0]$ there exists an analytic function $f = (f_1, \cdots, f_N): \D \rightarrow \C^N$ such that $$f(0) = x \text{ and } \underset{\gamma \in \T}{\esssup} \n f(\gamma) \n_\gamma < \lambda.$$ By \eqref{compatible-lattice}, this implies in particular that $f$ is coordinate bounded. Since $ z \mapsto \log | f_k(z)|$ is subharmonic, we have $$\log  | x_k |  = \log | f_k(0)| \le \int \log | f_k(\gamma) | dm(\gamma), \text{ for } 1 \le k \le N. $$ Hence $ \log  | x | \le \int \log | f(\gamma) | dm(\gamma) $. Obviously, $ \int \log \n f(\gamma) \n_\gamma dm(\gamma) < \log \lambda$, whence $$ \inf  \int \log \n f(\gamma)\n_\gamma \, dm(\gamma) \le \log \n x\n_{\mathscr{L}[0]}.$$

Conversely, assume that $x \in \C^N$ and $ x_k \ne 0$ for all $1\le k \le N$ and let $f: \D \rightarrow \C^N$ be any coordinate bounded analytic function such that   $ \log | x | \le \int \log | f(\gamma) | dm(\gamma)$. Then by \eqref{compatible-lattice}, $\underset{\gamma \in \T}{\esssup} \n f(\gamma)\n_\gamma < \infty$ and there exists $y \in \C^N$ such that \begin{eqnarray}\label{xy}| x | \le | y| \text{ and } \log |y|  = \int \log | f(\gamma) | dm(\gamma) .\end{eqnarray} Define $u(\gamma): = \log | f(\gamma)| $. By assumption, $x_k\ne 0$ and $ f_k $ is bounded, hence $\log |f_k| \in L_1(\T)$, so we can define the Hilbert transform of $u_k$. Let $\tilde{u}(\gamma)$ be the Hilbert transform of $u(\gamma)$ and define $g(\gamma) = e^{u(\gamma) + i \tilde{u}(\gamma)}$. Then $g_k (\gamma)= e^{u_k(\gamma) + i \tilde{u}_k(\gamma)}$ is the boundary value of an outer function, hence $$\log |g_k(0)| = \int \log | g_k(\gamma) | dm(t) = \int u_k(\gamma) dm(\gamma) = \log | y_k| .$$ Thus $| y | = | g(0)|$. By \cite[Prop. 2.4]{CCRSW}, we have \begin{eqnarray*} \n y \n_{\mathscr{L}[0]} & = & \n g(0) \n_{\mathscr{L}[0]}\le \exp\Big(\int \log \n g(\gamma)\n_\gamma \,dm(\gamma)\Big)\nonumber \\ & = & \exp\Big(\int \log \n f(\gamma)\n_\gamma \,dm(t)\Big) .  \end{eqnarray*} It is easy to see that $\mathscr{L}[0]$ is a Banach lattice and by \eqref{xy},  $$\n x \n_{\mathscr{L}[0]} \le \n y \n_{\mathscr{L}[0]}. $$ Thus $$\log \n x \n_{\mathscr{L}[0]} \le \log \n y \n_{\mathscr{L}[0]} \le  \int \log \n f(\gamma) \n_\gamma \, dm(\gamma).$$ This proves the converse inequality.
\end{proof}

\begin{rem}
The preceding result should be compared with \cite[Cor. 5.2]{CCRSW}, where it is shown that $$\Big\{ L^{p_\gamma}(X, \Sigma, \mu): \gamma \in \T\Big\} [z] = L^{p_z}(X, \Sigma, \mu),$$ where $1/p_z = \int (1/p_\gamma) P_z(d\gamma)$.
\end{rem}

\begin{defn}
Let $\mathscr{L} = (\C^N, \n \cdot \n_{\mathscr{L}})$ be a symmetric Banach lattices, we define $S_{\mathscr{L}}$ to be the space of $N\times N$ matrices equipped with the norm : $$ \n A \n_{S_{\mathscr{L}}} = \n (s_1(A), \cdots, s_N(A)) \n_{\mathscr{L}}, $$ where $s_1(A), \cdots, S_N(A)$ are singular numbers of the matrix $A$.
\end{defn}

If the Banach lattices $\mathscr{L}_\gamma$ considered above are all symmetric, i.e., for any permutation $\sigma \in \mathfrak{S}_N$ and any $x_k \in \C$,  $$\n \sum_{k=1}^N x_k e_{\sigma(k)} \n_{\mathscr{L}_\gamma} =    \n \sum_{k=1}^N x_k e_k \n_{\mathscr{L}_\gamma},$$ then to each $\mathscr{L}_\gamma$ is associated a Schatten type space $S_{\mathscr{L}_\gamma} = (M_N, \n\cdot\n_{S_{\mathscr{L}_\gamma}})$. 

The following proposition is classical (c.f. \cite{Pietsch1}), we omit its proof.
\begin{prop}\label{schatten-inter}
Let $\{ \mathscr{L}_\gamma = (\C^N, \n \cdot\n_\gamma): \gamma \in \T\}$ be an interpolation family of compatible symmetric Banach lattices and consider the associated interpolation family: $$\{ S_{\mathscr{L}_\gamma} = (M_N, \n\cdot\n_{S_{\mathscr{L}_\gamma}}): \gamma \in \T\}.$$ Then for any $z \in \D$, we have the following isometric identification $$Id: S_{\mathscr{L}[z]} \rightarrow \{ S_{\mathscr{L}_\gamma}\}[z].$$
\end{prop}

Combining Propositions \ref{proplattice} and \ref{schatten-inter}, we have the following:
\begin{cor}
Consider the interpolation family $\{ S_{\mathscr{L}_\gamma}: \gamma \in \T\}.$ Let  $M: \T\rightarrow \T$ be measure preserving and let $\{ \widetilde{S}_{\mathscr{L}_\gamma}= S_{\mathscr{L}_{M(\gamma)}}: \gamma \in \T\}$, then $$Id: \{S_{\mathscr{L}_\gamma}\}[0] \rightarrow \{\widetilde{S}_{\mathscr{L}_\gamma}\}[0]$$ is isometric.

\end{cor}

The following proposition is related to our problem, see the discussion after it.
\begin{prop}\label{dual}
Let $\{ E_\gamma: \gamma \in \T\}$ be an interpolation family of $N$-dimensional spaces such that there exist $c, C> 0$, for any $x \in \C^N$, $$c\cdot \min_k |x_k| \le \n x \n_\gamma \le C \cdot \max_k |x_k| \text{ for } a.e.\, \gamma \in \T. $$ Assume that  $Id: E_{\bar{\gamma}} \rightarrow \overline{E_\gamma}^*$ is isometric  for $a.e. \, \gamma \in \T$. Then $$E[\zeta] = \ell_2^N, \text{ for any } \zeta \in (-1, 1).$$ 
\end{prop}
\begin{proof}
Fix $\zeta \in (-1, 1)$. For any $x \in \C^N$. Given any analytic function $f: \D \rightarrow \C^N$ such that $f(\zeta) =x$ and $\n f\n_{H^\infty(\{E_\gamma\})}< \infty$. Since $\zeta = \bar{\zeta}$, we have $f(\zeta) = f(\bar{\zeta}) = x$. The assumption on the interpolation family implies that the function $ z \mapsto \langle f(z), f(\bar{z}) \rangle $ is bounded analytic, hence \begin{eqnarray*} \log \n x\n_{\ell_2^N}^2 &= &\log|\langle f(\zeta), f(\bar{\zeta}) \rangle |\le \int \log | \langle f(\gamma), f(\bar{\gamma})  \rangle | P_\zeta(d\gamma)\\ &  \le & \int \log \Big(\n f(\gamma) \n_{E_\gamma} \n f(\bar{\gamma}) \n_{\overline{E_\gamma^*}}\Big)P_\zeta(d\gamma) \\ & = &  \int \log \Big( \n f(\gamma) \n_{E_\gamma} \n f(\bar{\gamma}) \n_{E_{\bar{\gamma}}}\Big)P_\zeta(d\gamma) \\ & \le &  \log \Big( \n f \n_{H^\infty(\{E_\gamma\})}^2\Big). \end{eqnarray*} Hence $\n x\n_{\ell_2^N} \le \n f \n_{H^\infty(\{E_\gamma\})}.$ It follows that $$ \n x\n_{\ell_2^N} \le \n x\n_{E[\zeta]}.$$ By duality, this inequality also holds in the dual case, hence we must have $\n x\n_{\ell_2^N} = \n x\n_{E[\zeta]}.$
\end{proof}

Let $Q_j$ be the open arc of $\T$ in the $j$-th quadrant, i.e., $$Q_j = \Big\{ e^{i\theta}: \frac{(k-1)\pi}{2}< \theta<  \frac{k \pi}{2}\Big\} \text{ for }  1 \le j \le 4.$$ Suppose that $X$ and $Y$ are $N$-dimensional, define two  interpolation families $\{Z_\gamma: \gamma \in \T\}$ and $\{ \widetilde{Z}_\gamma: \gamma \in \T\}$  by letting $$Z_\gamma =\left\{ \begin{array}{cl} X, & \gamma \in Q_1 \\ Y, & \gamma \in Q_2 \\ \overline{Y^*}, & \gamma \in Q_3 \\ \overline{X^*}, & \gamma \in Q_4\end{array},\right. \quad \widetilde{Z}_\gamma =\left\{ \begin{array}{cl} X, & \gamma \in Q_1 \\ Y, & \gamma \in Q_2 \\ \overline{X^*}, & \gamma \in Q_3 \\ \overline{Y^*}, & \gamma \in Q_4\end{array}.\right. $$ By Proposition \ref{dual}, $Z[0] = \ell_2^N$. For suitable choices of $X$ and $Y$, we could have $\widetilde{Z}[0] \ne \ell_2^N$. More precisely, we have the following proposition.

\begin{prop}
For any $\alpha \in \T$, define a $2 \times 2$ selfadjoint matrix $$\delta_{\alpha}: = \left[\begin{array}{cc}0 & \overline{\alpha}\\ \alpha & 0 \end{array}\right]. $$ For $0< \varepsilon < 1$, let $w^{\alpha,\varepsilon} = ( I + \varepsilon \delta_\alpha)^{1/2}$ and $X = \ell^2_{w^{\alpha,\varepsilon } }$. Consider the weight $W^{\alpha, \varepsilon}$ and the interpolation family generated by it as follows: $$W^{\alpha, \varepsilon}(\gamma) =\left\{ \begin{array}{cl} I + \varepsilon \delta_\alpha, & \gamma \in Q_1 \\ (I + \varepsilon \overline{ \delta_\alpha})^{-1}, & \gamma \in Q_2 \\ (I + \varepsilon \delta_\alpha)^{-1}, & \gamma \in Q_3 \\ I   + \varepsilon \overline{\delta_\alpha}, & \gamma \in Q_4\end{array};\right. \quad \widetilde{Z^{\alpha, \varepsilon}_\gamma} =\left\{ \begin{array}{cl} X, & \gamma \in Q_1 \\ X^*, & \gamma \in Q_2 \\ \overline{ X}^*, & \gamma \in Q_3 \\ \overline{ X }, & \gamma \in Q_4\end{array}.\right.  $$ 
There exists $\alpha \in \T$ and $0 < \varepsilon_0< 1$, such that if $ 0 < \varepsilon < \varepsilon_0$ then  $$\widetilde{Z^{\alpha, \varepsilon}}[0] \ne \ell_2^N.$$ 
\end{prop}
\begin{proof}
We have $$W^{\alpha, \varepsilon}(\gamma) =\left\{ \begin{array}{cl} I + \varepsilon \delta_\alpha, & \gamma \in Q_1 \\ I - \varepsilon \overline{ \delta_\alpha} + \varepsilon^2 I + \mathcal{O}(\varepsilon^3), & \gamma \in Q_2 \\ I -\varepsilon \delta_\alpha + \varepsilon^2 I + \mathcal{O}(\varepsilon^3), & \gamma \in Q_3 \\ I   + \varepsilon \overline{\delta_\alpha}, & \gamma \in Q_4\end{array}.\right.$$ Applying a slightly modified variant of the approximation equation \eqref{approximation}, we have \begin{eqnarray*}| F_{W^{\alpha, \varepsilon}} (0) |^2  & = & I + \frac{\varepsilon^2 I }{2} - \varepsilon^2 \left[\begin{array}{cc} \n P_{+}(h_\alpha)\n_2^2 & 0 \\ 0 &\n P_{+}(\overline{h_\alpha})\n_2^2 \end{array}\right]  + \mathcal{O}(\varepsilon^3);\end{eqnarray*} where $h_\alpha = \alpha 1_{Q_1}- \overline{\alpha}1_{Q_2} - \alpha 1_{Q_3} + \overline{\alpha} 1_{Q_4}.$ 

Assume by contradiction that $\widetilde{Z^{\alpha, \varepsilon}}[0] = \ell_2^N$ for any $\alpha\in \T$ and small $\varepsilon$. Then we must have $\n P_{+}(h_{\alpha})\n_2^2  = \frac{1}{2}$ for any $\alpha \in \T$. In particular, $$\alpha \mapsto \n P_{+}(h_\alpha)\n_2^2 \text{ is a constant function on $\T$}.$$ It follows that the following function is a constant function:  \begin{eqnarray*}\label{constantfunction} C(\alpha) &= &  \Re \langle \alpha P_{+}(1_{Q_1}), - \overline{\alpha} P_{+}(1_{Q_2}) \rangle + \Re \langle \alpha P_{+}(1_{Q_1}), \overline{\alpha} P_{+}(1_{Q_4}) \rangle  \nonumber \\ & +&  \Re \langle -\overline{\alpha} P_{+}(1_{Q_2}), - \alpha P_{+}(1_{Q_3}) \rangle + \Re \langle - \alpha P_{+}(1_{Q_3}),  \overline{\alpha} P_{+}(1_{Q_4}) \rangle .   \end{eqnarray*} Clearly, by translation invariance of Haar measure, we have $$\langle P_{+}(1_{Q_1}), P_{+}(1_{Q_2}) \rangle =  \langle P_{+}(1_{Q_2}), P_{+}(1_{Q_3}) \rangle  = \langle P_{+}(1_{Q_3}), P_{+}(1_{Q_4}) \rangle,$$   $$ \langle P_{+}(1_{Q_1}), P_{+}(1_{Q_4}) = \langle P_{+}(1_{Q_2}), P_{+}(1_{Q_1}) \rangle,$$   hence $$ C(\alpha) = - \Re\Big\{2\alpha^2 \Big( \langle P_{+}(1_{Q_1}), P_{+}(1_{Q_2}) \rangle  - \overline{\langle P_{+}(1_{Q_1}), P_{+}(1_{Q_2}) \rangle }\Big)\Big\}. $$ Then $\alpha \mapsto C(\alpha)$ is constant function if and only if $$ \langle P_{+}(1_{Q_1}), P_{+}(1_{Q_2}) \rangle  - \overline{\langle P_{+}(1_{Q_1}), P_{+}(1_{Q_2})\rangle } = 0,$$ which is equivalent to \begin{eqnarray}\label{0-imaginary} \Im \Big(\langle P_{+}(1_{Q_1}), P_{+}(1_{Q_2}) \rangle\Big) = 0.\end{eqnarray} By a similar computation as in the proof of inequality \eqref{three-arc}, we have $$ \Im \Big(\langle P_{+}(1_{Q_1}), P_{+}(1_{Q_2}) \rangle\Big) = \frac{4}{\pi^2} \sum_{k = 0}^\infty \frac{2k+1}{(4k+1)^2(4k+3)^2},$$ this contradicts \eqref{0-imaginary}, and hence completes the proof.
\end{proof}

\section{Appendix}\label{appendix}
Here we reformulate the argument of \cite{CCRSW} to emphasize the crucial role played by a certain inner function associated to the measurable partition of the unit circle in proving Theorem \ref{known}. It follows from the preceding that the analogous inner function for a measurable partition into 3 subsets does not exist.


\begin{lem}\label{inner}
Suppose that $\Gamma_0 \cup \Gamma_1$ is a measurable partition of $\T$. Then there exists an inner function $\varphi$ such that $\varphi(0) = 0$. And $\varphi(\Gamma_0) \cup \varphi(\Gamma_1)$ is a partition of $\T$ into two disjoint arcs (up to negligible sets). Moreover, \begin{eqnarray}\label{m.p.} m(\varphi(\Gamma_0)) = m(\Gamma_0) \textit{ and } m(\varphi(\Gamma_1)) = m(\Gamma_1).\end{eqnarray}
\end{lem}
\begin{proof}
Since any origin-preserving inner function $\varphi$ preserves the measure $m$ on $\T$ (indeed note $\int_\T \varphi(\gamma)^n dm(\gamma) = \int_\T \gamma^n dm(\gamma)\, \forall n \in \Z$), it suffices to show the existence of an inner function satisfying the partition condition. 

Let $v = 1_{\Gamma_1}: \T \rightarrow \R$ be the characteristic function of $\Gamma_1$, its harmonic extension on $\D$ will also be denoted by $v$. Note that $0 < v(z) < 1$ for any $z \in\D$. Let $\tilde{v}$ be the harmonic conjugate of $v$ and define $\psi = v + i \tilde{v}$ on $\D$. Then $\psi$ is an analytic map from $\D$ to $\mathcal{S}: = \{ z \in \C: 0 < \Re (z) < 1 \}$ and has non-tangential limit $\psi(\gamma) = v(\gamma) + i \tilde{v}(\gamma)$, a.e. $\gamma \in \T$. Thus $$\psi(\Gamma_0) \subset \partial_0 \text{ and } \psi(\Gamma_1) \subset \partial_1,$$ where $\partial_0= \{ z \in \C : \Re(z)  = 0\}$ and $\partial_1 = \{ z \in \C: \Re(z) = 1\}$. Let $\tau: \mathcal{S} \rightarrow \D$ be a Riemann conformal mapping such that $\tau(\psi(0))= 0$.   Note that $\tau(\partial_0)$ and  $\tau(\partial_1)$ are disjoint open arcs of $\T$. Define $\varphi= \tau\circ \psi: \D \rightarrow \D.$ Then $\varphi$ is an inner function such that $\varphi(0) = 0$. We have $$ \varphi(\Gamma_0) \subset \tau(\partial_0) \text{ and } \varphi(\Gamma_1) \subset \tau(\partial_1).$$ Hence $m(\varphi(\Gamma_0)) \le m(\tau(\partial_0))$ and $ m(\varphi(\Gamma_1)) \le m(\tau(\partial_1))$. Since $\varphi$ preserves the measure $m$, we have $$1 = m(\varphi(\Gamma_0)) + m(\varphi(\Gamma_1)) \le m(\tau(\partial_0)) + m(\tau(\partial_1)) = 1.$$ Thus up to negligible sets, we have $$\varphi(\Gamma_0) = \tau(\partial_0) \text{ and } \varphi(\Gamma_1) = \tau(\partial_1).$$
\end{proof}

\begin{proof}[Proof of Theorem \ref{known}]
Suppose $\Gamma_0 \cup \Gamma_1$ is a measurable partition of the circle and let the interpolation family $\{ X_\gamma: \gamma \in \T\}$ be such that $$X_\gamma = Z_0 \text{ for all } \gamma \in \Gamma_0,\,\, X_\gamma = Z_1\text{ for all } \gamma \in \Gamma_1.$$
By Lemma \ref{inner}, we can find an inner function $\varphi$ such that $\varphi(0) = 0$ and $\varphi(\Gamma_0) = J_0, \, \,\varphi(\Gamma_1) = J_1$ up to negligible sets, where $J_0 \cup J_1$ is a partition of the circle into disjoint arcs. Consider the interpolation family of spaces $\{ \widetilde{X}_\gamma: \gamma\in \T\}$ such that $$\widetilde{X}_\gamma = Z_0 \text{ for all } \gamma \in J_0,\,\, \widetilde{X}_\gamma = Z_1\text{ for all } \gamma \in J_1.$$ Then by a conformal mapping, it is easy to see \begin{eqnarray}\label{inter-pair}\widetilde{X}[0] = (Z_0, Z_1)_\theta, \, \, \theta = m(J_1) = m(\Gamma_1).\end{eqnarray} We have $\widetilde{X}_{\varphi(\gamma)} = X_\gamma$ for a.e. $\gamma \in \T$. If $x \in \C^N$ is such that $\n x \n_{\widetilde{X}[0]} < 1$, then by definition, there exists an analytic function $f: \T\rightarrow \C^N$ such that $f(0) = x$ and  $\underset{t \in \T}{\esssup} \n f(\gamma)\n_{\widetilde{X}_\gamma}<1$. Thus $$\underset{\gamma \in \T}{\esssup}\n (f\circ \varphi)(\gamma)\n_{X_{t}} = \underset{\gamma \in \T}{\esssup}\n (f\circ \varphi)(\gamma)\n_{\widetilde{X}_{\varphi(\gamma)}}   = \underset{\gamma \in \T}{\esssup} \n f(\gamma)\n_{\widetilde{X}_\gamma}<1.$$ Since $(f \circ \varphi)(0) = f(0) = x$, the above inequality shows that $\n x\n_{X[0]}< 1.$ By homogeneity, $\n x \n_{X[0]} \le \n x \n_{\widetilde{X}[0]} .$ But if we consider the dual of the above interpolation family, then we get the same inequality, hence we must have \begin{eqnarray}\label{origin-pair}\n x \n_{X[0]} = \n x\n_{\widetilde{X}[0]}.\end{eqnarray} By \eqref{origin-pair} and \eqref{inter-pair}, we have $$X[0] = (Z_0, Z_1)_\theta, \, \, \theta = m(\Gamma_1).$$
\end{proof}

By definition, a space is arcwise $\theta$-Hilbertian if it can be obtained by complex interpolation of a family of spaces on the circle such that on an arc, the spaces are Hilbertian.

 \begin{rem}[Communicated by Gilles Pisier] The preceding argument also shows that, as conjectured in \cite{Pisier8}, of which we use the terminology, any   $\theta$-Hilbertian  Banach space is automatically arcwise $\theta$-Hilbertian, at least under suitable assumptions on the dual spaces, that are automatic in the finite dimensional case. We merely indicate the argument in the latter case. Consider a  measurable partition $\Gamma_0\cup \Gamma_1$ of the unit circle
 with $m(\Gamma_1)=\theta$ and a family of $n$-dimensional spaces $\{E_\gamma\mid \gamma\in \partial D\}$ such that
 $E_\gamma=\ell_2^n$ for any $\gamma\in \Gamma_1$ but    $E_\gamma$  is  arbitrary for $\gamma \in \Gamma_0$. If $\varphi$ is the inner function appearing in Lemma \ref{inner},
 and if we set $F_\gamma=E_{\varphi(\gamma)}$ then   the identity map
 $Id: E[0]\to F[0]$ is clearly contractive and $F[0]$ is arcwise $\theta$-Hilbertian.
 Applying this to  the dual family $\{E_\gamma^* \}$ in place of 
 $\{E_\gamma\}$  and using the duality theorem from \cite[Th. 2.12 ]{CCRSW}) we find that  $Id:\ {E[0]}^*\to {F[0]}^*$ is also contractive, and hence
 is isometric. This shows that $E[0]$ is arcwise $\theta$-Hilbertian.
\end{rem}

\section*{Acknowledgements}
The author is grateful to  Gilles Pisier for stimulating discussions and valuable suggestions, he would like to thank the referee for careful reading of the manuscript. The author was partially supported by the ANR grant 2011-BS01-00801 and the A*MIDEX grant.

\newcommand{\etalchar}[1]{$^{#1}$}
\def\cprime{$'$}
\providecommand{\bysame}{\leavevmode\hbox to3em{\hrulefill}\thinspace}
\providecommand{\MR}{\relax\ifhmode\unskip\space\fi MR }
\providecommand{\MRhref}[2]{%
  \href{http://www.ams.org/mathscinet-getitem?mr=#1}{#2}
}
\providecommand{\href}[2]{#2}

\flushright{Yanqi QIU}
\flushright{Institut Mathématique de Jussieu, \'Equipe d'Analyse Fonctionnelle}
\flushright{Universit\'e Paris VI}
\flushright {Place Jussieu, 75252 Paris Cedex 05, France}
\flushright{yqi.qiu@gmail.com}

\end{document}